\documentclass[leqno,12pt]{amsart}
\usepackage{amsmath,amstext,amssymb,amsopn,amsthm,mathrsfs}
\usepackage{epsfig,graphics,latexsym,graphicx,color}
\definecolor{darkgreen}{rgb}{0,.5,0}

\numberwithin{equation}{section}

\allowdisplaybreaks

\newtheorem{theorem}{Theorem}[section]

\newtheorem{corollary}[theorem]{Crolorrary}
\newtheorem{lemma}{Lemma}[section]
\newtheorem*{rem*}{Remark}
\newtheorem{remark}{Remark}[section]

\begin{document}

\title[]{Characterizations of weighted BMO space and its application}

\author[]{Dinghuai Wang, Jiang Zhou$^\ast$ and Zhidong Teng}
\address{College of Mathematics and System Sciences, Xinjiang University,  Urumqi 830046 \endgraf
         Republic of China}
\email{Wangdh1990@126.com; zhoujiang@xju.edu.cn; zhidong1960@163.com}
\thanks{The research was supported by National Natural Science Foundation
of China (Grant No.11661075 and No. 11271312). \\ \qquad * Corresponding author, zhoujiang@xju.edu.cn.}

\begin{abstract}
In this paper, we prove that the weighted BMO space as follows
$${\rm BMO}^{p}(\omega)=\Big\{f\in L^{1}_{\rm loc}:\sup_{Q}\|\chi_{Q}\|^{-1}_{L^{p}(\omega)}\big\|(f-f_{Q})\omega^{-1}\chi_{Q}\big\|_{L^{p}(\omega)}<\infty\Big\}$$
is independent of the scale $p\in (0,\infty)$ in sense of norm when $\omega\in A_{1}$. Moreover, we can replace $L^{p}(\omega)$ by $L^{p,\infty}(\omega)$. As an application, we characterize this space by the boundedness of the bilinear commutators $[b,T]_{j} (j=1,2)$, generated by the bilinear convolution type Calder\'{o}n-Zygmund operators and the symbol $b$, from $L^{p_{1}}(\omega)\times L^{p_{2}}(\omega)$ to $L^{p}(\omega^{1-p})$ with $1<p_{1},p_{2}<\infty$, $1/p=1/p_{1}+1/p_{2}$ and $\omega\in A_{1}$. Thus we answer the open problem proposed in \cite{C} affirmatively.
\end{abstract}
\maketitle

\maketitle

\section{Introduction}
A locally integrable function $f$ is said to belong to \rm{BMO} space if there exists a constant
$C > 0$ such that for any cube $Q\subset \mathbb{R}^n$,
$$\frac{1}{|Q|}\int_{Q}|f(x)-f_{Q}|dx\leq C,$$
where $f_{Q}=\frac{1}{|Q|}\int_{Q}f(x)dx$ and the minimal constant $C$ is defined by $\|f\|_{*}$.

There are a number of classical results that demonstrate ${\rm BMO}$ functions are the right collections to do harmonic analysis on the boundedness of commutators. A well known result of Coifman, Rochberg and Weiss \cite{CRW} states that the commutator
$$[b,T](f)=bT(f)-T(bf)$$
is bounded on some $L^p$, $1<p<\infty$, if and only if $b\in \mathrm{BMO}$, where $T$ is the Hilbert transform.  Janson extended the result in \cite{J} via the commutators of Calder\'{o}n-Zygmund operators with smooth homogeneous kernels; Chanillo in \cite{Ch} did the same for commutators of the fractional integral operator with the restriction that $n-\alpha$ be an even integer. The theory was then extended and generalized to several directions. For instance, Bloom \cite{B} investigated the same result in the weighted setting; Uchiyama \cite{U} extended the boundednss results on the commutator to compactness; Krantz and Li in \cite{KL1} and \cite{KL2} have applied commutator theory to give a compactness characterization of Hankel operators on holomorphic Hardy spaces $H^{2}(D)$, where $D$ is a bounded, strictly pseudoconvex domain in $\mathbb{C}^n$. It is perhaps for this important reason that the boundedness of $[b, T]$ attracted one¡¯s attention among researchers in PDEs.

Recently, Chaffee \cite{C} considered the multilinear setting and proved that for $0\leq \alpha<2n$, $1<p_{1},p_{2}<\infty$ and
$$\frac{1}{p_{1}}+\frac{1}{p_{2}}-\frac{\alpha}{n}=\frac{1}{q},$$
if $q>1$, then
\begin{equation}\label{1}
[b,T]_{j}: L^{p_{1}}\times L^{p_{2}}\rightarrow L^{q}\Leftrightarrow b\in {\rm BMO}
\end{equation}
for $j=1,2$, where $T$ is a bilinear operator of convolution type with a homogeneous kernel of degree $-2n+\alpha$. In his proof required the use of H\"{o}lder inequality with $q$ and $q'$, the exponent $q$ must be larger than $1$. Thus, he asked
\vspace{0.3cm}

\textbf{Problem 1.} If $\frac{1}{2}<q<1$ and $[b,T]_{j}$ is a bounded operator from $L^{p_{1}}\times L^{p_{2}}$ to $L^{q}$, is $b$ in ${\rm BMO}$ space?

At the same time, Wang, Jiang and Pan \cite{WJP} obtain the similar result as \eqref{1} for bilinear fractional integral operator, they also asked
\vspace{0.3cm}

\textbf{Problem 2.} If $\vec{b}=(b_{1},b_{2})$ and $[\Pi\vec{b},I_{\alpha}]$ is a bounded operator from $L^{p_{1}}\times L^{p_{2}}$ to $L^{q}$, is $\vec{b} \in {\rm BMO}\times {\rm BMO}$?

In this paper, we will give an answer of Problem 1 and show that the answer of Problem 2 is affirmative for the case ${\vec{b}=(b,b)}$, using the following Theorem \ref{main1} and Theorem \ref{main2}. We focus on proving the case $\alpha=0$. For $0<\alpha<2n$, a similar arguments are applied with necessary modifications, one can obtain the desired result. Moreover, we extend the result to weighted case. To state our result, we first give the following denotations.

We recall the definition of $A_{p}$ weight introduced by Muckenhoupt \cite{M}. For $1< p<\infty$ and a nonnegative locally integrable function $\omega$ on $\mathbb{R}^n$, $\omega$ is in the
Muckenhoupt $A_{p}$ class if it satisfies the condition
$$[\omega]_{A_{p}}:=\sup_{Q}\bigg(\frac{1}{|Q|}\int_{Q}\omega(x)dx\bigg)\bigg(\frac{1}{|Q|}\int_{Q}\omega(x)^{-\frac{1}{p-1}}dx\bigg)^{p-1}<\infty.$$
And a weight function $\omega$ belongs to the class $A_{1}$ if
$$[\omega]_{A_{1}}:=\frac{1}{|Q|}\int_{Q}\omega(x)dx\Big(\mathop\mathrm{ess~sup}_{x\in Q}\omega(x)^{-1}\Big)<\infty.$$
We write $A_{\infty}=\bigcup_{1\leq p<\infty}A_{p}$.

Let $\omega\in A_{\infty}$ and $p\in (0,\infty)$. We let $L^{p}(\omega)$ be the space of all measurable functions $f$ such that
$$\|f\|_{L^{p}(\omega)}:=\bigg(\int_{\mathbb{R}^n}|f(x)|^{p}\omega(x)dx\bigg)^{1/p}<\infty.$$

Let $0< p<\infty$. Given a nonnegative locally integrable function $\omega$, the weighted $\mathrm{BMO}$ space $\mathrm{BMO}^{p}(\omega)$ is defined by the set of all functions $f\in L^{1}_{\mathrm{loc}}(\mathbb{R}^{n})$ such that
\begin{eqnarray*}
\|f\|_{{\rm BMO}^{p}(\omega)}&=&\sup_{Q}\bigg(\frac{1}{\omega(Q)}\int_{Q}|f(y)-f_{Q}|^{p}\omega(y)^{1-p}dy\bigg)^{1/p}\\
&=&\sup_{Q}\frac{1}{\|\chi_{Q}\|_{L^{p}(\omega)}}\big\|\frac{(f-f_{Q})\chi_{Q}}{\omega}\big\|_{L^{p}(\omega)}<\infty,
\end{eqnarray*}
where $\omega(Q)=\int_{Q}\omega(x)dx$. We write $\mathrm{BMO}^{1}(\omega)=\mathrm{BMO}(\omega)$ simple. In \cite{G}, Garc\'{i}a-Cuerva proved that if $\omega\in A_{1}$, $\mathrm{BMO}(\omega)=\mathrm{BMO}^{p}(\omega)$ for $1< p<\infty$ with equivalence of the corresponding norms.
\vspace{0.3cm}

\textbf{Problem 3.} Let $X$ be a quasi-Banach function space and $\omega\in A_{1}$. Is the norm $\|f\|_{\mathrm{BMO}(\omega)}$ equivalent to
$$\|f\|_{{\rm BMO}_{X}(\omega)}=\sup_{Q}\frac{1}{\|\chi_{Q}\|_{X}}\big\|\frac{(f-f_{Q})\chi_{Q}}{\omega}\big\|_{X}<\infty?$$
The aim of this paper is to show that the answer of Problem 3 is affirmative for $X=L^{p,\infty}(\omega)(1<p<\infty)$ and $X=L^{r}(\omega)(0<r<1)$ .

\begin{theorem}\label{main1}
Let $\omega\in A_{1}$ and $X=L^{p,\infty}(\omega)$ with $1<p<\infty$. Then
$${\rm BMO}(\omega)={\rm BMO}_{X}(\omega)$$
with equivalence of the corresponding norms.
\end{theorem}

\begin{theorem}\label{main2}
Let $\omega\in A_{1}$ and $0<r<1$. Then
$${\rm BMO}(\omega)={\rm BMO}^{r}(\omega)$$
with equivalence of the corresponding norms.
\end{theorem}

\begin{remark}
In the unweighted setting, Str\"{o}mberg in \cite{St} showed that for $0<s\leq \frac{1}{2}$, $p>0$, there exists a constant $C$ such that
$$s^{1/p}\|f\|_{{\rm BMO}^{*}_{s}}\leq \|f\|_{{\rm BMO}^{p}}\leq C\|f\|_{{\rm BMO}_{s}^{*}},$$
where
$$\|f\|_{{\rm BMO}_{s}^{*}}=\sup_{Q}\inf_{c}\inf\big\{t\geq 0:\big|\{x\in Q:|f(x)-c|>t\}\big|<s|Q|\big\}.$$
\end{remark}

Recall that bilinear singular integral operator $T$ is a bounded operator which satisfies
$$\|T(f_{1},f_{2})\|_{L^{p}}\leq C\|f_{1}\|_{L^{p_{1}}}\|f_{2}\|_{L^{p_{2}}},$$
for some $1<p_{1},p_{2}<\infty$ with $1/p=1/p_{1}+1/p_{2}$ and the function $K$, defined off the diagonal $y_{0}=y_{1}=y_{2}$ in $(\mathbb{R}^{n})^{2+1}$, satisfies the conditions as follow:

(1) The function $K$ satisfies the size condition.
$$|K(x,y_{1},y_{2})|
\leq\frac{C}{\big(|x-y_{1}|+|x-y_{2}|\big)^{2n}};$$

(2) The function $K$ satisfies the regularity condition. For some $\gamma>0$,
if $|y_{1}-y'_{1}|\leq \frac{1}{2}\max\{|x-y_{1}|,|x-y_{2}|\}$
$$|K(x,y_{1},y_{2})-K(x',y_{1},y_{2})|
\leq\frac{C|y_{1}-y'_{1}|^{\gamma}}{\big(|x-y_{1}|+|x-y_{2}|\big)^{2n+\gamma}};$$
if $|y_{2}-y'_{2}|\leq \frac{1}{2}\max\{|x-y_{1}|,|x-y_{2}|\}$
$$|K(x,y_{1},y_{2})-K(x',y_{1},y_{2})|
\leq\frac{C|y_{2}-y'_{2}|^{\gamma}}{\big(|x-y_{1}|+|x-y_{2}|\big)^{2n+\gamma}}.$$

Then we say $K$ is a bilinear Calder\'{o}n-Zygmund kernel. If $x\notin {\rm supp} f_{1}\bigcap {\rm supp} f_{2}$, then
$$T(f_{1},f_{2})(x)=\int_{\mathbb{R}^n}\int_{\mathbb{R}^n}K(x,y_{1},y_{2})f_{1}(y_{1}) f_{2}(y_{2})dy_{1}dy_{2}.$$

The linear commutators are defined by
$$[b,T]_{1}(f_{1},f_{2})(x):=b(x)T(f_{1},f_{2})(x)-T(bf_{1},f_{2})(x),$$
and
$$[b,T]_{2}(f_{1},f_{2})(x):=b(x)T(f_{1},f_{2})(x)-T(f_{1},bf_{2})(x).$$

The iterated commutator is defined by
$$[\Pi\vec{b},T](f_{1},f_{2})(x):=[b_{2},[b_{1},T]_{1}]_{2}(f_{1},f_{2})(x).$$

In this paper, we say that an operator is of 'convolution type' if the kernel $K(x,y_{1},y_{2})$ is actually of the form $K(x-y_{1},x-y_{2})$. The applications of Theorem \ref{main1} and Theorem \ref{main2} as follows.
\begin{theorem}\label{t1}
Let $\omega\in A_{1}$, $\vec{b}=(b,b)$ and $T$ be a bilinear convolution type operator defined by
$$T(f_{1},f_{2})(x)=\int_{\mathbb{R}^n}\int_{\mathbb{R}^n}K(x-y_{1},x-y_{2})f_{1}(y_{1})f_{2}(y_{2})dy_{1}dy_{2}$$
for all $x\notin {\rm supp} f_{1}\bigcap {\rm supp} f_{2}$, where $K$ is a bilinear Calder\'{o}n-Zygmund kernel and such that for any cube $Q\subset \mathbb{R}^{2n}$ with $0\notin Q$, the Fourier series of $\frac{1}{K}$ is absolutely convergent. For $1<p_{1},p_{2}<\infty$ with $1/p=1/p_{1}+1/p_{2}$, the following statements are equivalent:
\begin{enumerate}
\item [\rm(a1)] $b\in \mathrm{BMO}(\omega)$;
\item [\rm(a2)] There exists a positive constant $C$ such that for $j=1,2$,
$$\|[b,T]_{j}(f_{1},f_{2})\cdot\omega^{-1}\|_{L^{p}(\omega)}\leq C\|f_{1}\|_{L^{p_{1}}(\omega)}\|f_{2}\|_{L^{p_{2}}(\omega)}.$$
\item [\rm(a3)] There exists a positive constant $C$ such that for $j=1,2$,
$$\|[b,T]_{j}(f_{1},f_{2})\cdot\omega^{-1}\|_{L^{p,\infty}(\omega)}\leq C\|f_{1}\|_{L^{p_{1}}(\omega)}\|f_{2}\|_{L^{p_{2}}(\omega)}.$$
\item [\rm(a4)] There exists a positive constant $C$ such that
$$\|[\Pi\vec{b},T](f_{1},f_{2})\cdot\omega^{-2}\|_{L^{p}(\omega)}\leq C\|f_{1}\|_{L^{p_{1}}(\omega)}\|f_{2}\|_{L^{p_{2}}(\omega)}.$$
\item [\rm(a5)] There exists a positive constant $C$ such that
$$\|[\Pi\vec{b},T](f_{1},f_{2})\cdot\omega^{-2}\|_{L^{p,\infty}(\omega)}\leq C\|f_{1}\|_{L^{p_{1}}(\omega)}\|f_{2}\|_{L^{p_{2}}(\omega)}.$$
\end{enumerate}
\end{theorem}
\vspace{0.3cm}

Specially, if $\omega(x)\equiv 1$, we have
\begin{corollary}
Let $\vec{b}=(b,b)$ and $T$ be a bilinear convolution type operator defined by
$$T(f_{1},f_{2})(x)=\int_{\mathbb{R}^n}\int_{\mathbb{R}^n}K(x-y_{1},x-y_{2})f_{1}(y_{1})f_{2}(y_{2})dy_{1}dy_{2}$$
for all $x\notin {\rm supp} f_{1}\bigcap {\rm supp} f_{2}$, where $K$ is a bilinear Calder\'{o}n-Zygmund kernel and such that for any cube $Q\subset \mathbb{R}^{2n}$ with $0\notin Q$, the Fourier series of $\frac{1}{K}$ is absolutely convergent. For $1<p_{1},p_{2}<\infty$ with $1/p=1/p_{1}+1/p_{2}$, the following statements are equivalent:
\begin{enumerate}
\item [\rm(b1)] $b\in \mathrm{BMO}$;
\item [\rm(b2)] $[b,T]_{j}$ is a bounded operator from $L^{p_{1}}\times L^{p_{2}}$ to $L^{p}$ for $j=1,2$;
\item [\rm(b3)] $[b,T]_{j}$ is a bounded operator from $L^{p_{1}}\times L^{p_{2}}$ to $L^{p,\infty}$ for $j=1,2$;
\item [\rm(b4)] $[\Pi\vec{b},T]$ is a bounded operator from $L^{p_{1}}\times L^{p_{2}}$ to $L^{p}$;
\item [\rm(b5)] $[\Pi\vec{b},T]$ is a bounded operator from $L^{p_{1}}\times L^{p_{2}}$ to $L^{p,\infty}$.
\end{enumerate}
\end{corollary}

A same argument we also have the following result.
\begin{corollary}
Let $\vec{b}=(b,b)$ and $I_{\alpha}$ be a bilinear fractional integral operator defined by
$$I_{\alpha}(f_{1},f_{2})(x)=\int_{\mathbb{R}^n}\int_{\mathbb{R}^n}\frac{f_{1}(y_{1})f_{2}(y_{2})}{\big(|x-y_{1}|+|x-y_{2}|\big)^{2n-\alpha}}dy_{1}dy_{2}.$$
For $0<\alpha<2n$, $1<p_{1},p_{2}<\infty$ with $1/q=1/p_{1}+1/p_{2}-\alpha/n$, the following statements are equivalent:
\begin{enumerate}
\item [\rm(c1)] $b\in \mathrm{BMO}$;
\item [\rm(c2)] $[b,I_{\alpha}]_{j}$ is a bounded operator from $L^{p_{1}}\times L^{p_{2}}$ to $L^{q}$ for $j=1,2$;
\item [\rm(c3)] $[b,I_{\alpha}]_{j}$ is a bounded operator from $L^{p_{1}}\times L^{p_{2}}$ to $L^{q,\infty}$ for $j=1,2$;
\item [\rm(c4)] $[\Pi\vec{b},I_{\alpha}]$ is a bounded operator from $L^{p_{1}}\times L^{p_{2}}$ to $L^{q}$;
\item [\rm(c5)] $[\Pi\vec{b},I_{\alpha}]$ is a bounded operator from $L^{p_{1}}\times L^{p_{2}}$ to $L^{q,\infty}$.
\end{enumerate}
\end{corollary}

Finally, two open problems will be given.
\vspace{0.3cm}

\textbf{Problem A.} Let $\vec{b}=(b_{1},b_{2})$ with $b_{1}\neq b_{2}$ and $[\Sigma\vec{b},T]:=[b_{1},T]_{1}+[b_{2},T]_{2}$. If $[\Sigma\vec{b},T]$ is a bounded operator from $L^{p_{1}}\times L^{p_{2}}$ to $L^{q}$, is $\vec{b}$ in ${\rm BMO}\times {\rm BMO}$?
\vspace{0.3cm}

\textbf{Problem B.} Let $\vec{b}=(b_{1},b_{2})$ with $b_{1}\neq b_{2}$. If $[\Pi\vec{b},T]$ is a bounded operator from $L^{p_{1}}\times L^{p_{2}}$ to $L^{q}$, is $\vec{b}$ in ${\rm BMO}\times {\rm BMO}$?
\vspace{0.3cm}

\section{Main Lemmas}

Throughout this paper, the letter $C$ denotes constants which are independent of main variables and may change from one occurrence to another. $Q(x,r)$ denotes a cube centered at $x$, with side length $r$, sides parallel to the axes.

For $X=L^{q_{2},\infty}(\omega)$, it is clear that ${\rm BMO}^{q_{1}}(\omega)$ is contained in ${\rm BMO}_{X}(\omega)$ and $\|\cdot\|_{{\rm BMO}_{X}(\omega)}\leq \|\cdot\|_{{\rm BMO}^{q_{2}}(\omega)} \leq \|\cdot\|_{{\rm BMO}^{q_{1}}(\omega)}$ if $1<q_{2}\leq q_{1}<\infty$. However, for $1<q_{1}< q_{2}<\infty$, one has the reverse inequality as follows.
\begin{lemma}\label{lem1}
Let $1< q_{1}< q_{2}<\infty$, $\omega\in A_{\infty}$ and $X=L^{q_{2},\infty}(\omega)$. Then ${\rm BMO}_{X}(\omega)$ is contained in $ {\rm BMO}^{q_{1}}(\omega)$ and $\|\cdot\|_{{\rm BMO}^{q_{1}}(\omega)}\leq C\|\cdot\|_{{\rm BMO}_{X}(\omega)}$.
\end{lemma}
\begin{proof}
Let $f\in {\rm BMO}_{X}(\omega)$. Given a fixed cube $Q \subset \mathbb{R}^n$, it is easy to see that $\|\chi_{Q}\|_{L^{q,\infty}(\omega)}=\omega(Q)^{1/q}$, then for any $\lambda>0$,
$$\frac{1}{\omega(Q)^{1/q_{2}}}\Big(\lambda^{q_{2}}\omega\big\{x\in Q:|f(x)-f_{Q}|>\lambda\omega(x)\big\}\Big)^{1/q_{2}}\leq \|f\|_{{\rm BMO}_{X}(\omega)};$$
that is,
$$\omega\big\{x\in Q:|f(x)-f_{Q}|>\lambda\omega(x)\big\}\leq \|f\|^{q_{2}}_{{\rm BMO}_{X}(\omega)}\omega(Q)\lambda^{-q_{2}}.$$
Choose
$$N=\|f\|_{{\rm BMO}_{X}(\omega)}\Big(\frac{q_{1}}{q_{2}-q_{1}}\Big)^{1/q_{2}}.$$
Thus,
\begin{eqnarray*}
\int_{Q}|f(x)-f_{Q}|^{q_{1}}\omega(x)^{1-q_{1}}dx&=&
\int_{Q}\Big(\frac{|f(x)-f_{Q}|}{\omega(x)}\Big)^{q_{1}}\omega(x)dx\\
&=&q_{1}\int_{0}^{\infty}\lambda^{q_{1}-1}\omega\big\{x\in Q:|f(x)-f_{Q}|>\lambda\omega(x)\big\}d\lambda\\
&\leq&q_{1}\int_{0}^{N}\lambda^{q_{1}-1}\omega(Q)d\lambda+q_{1}\int_{N}^{\infty}\lambda^{q_{1}-1}\|f\|^{q_{2}}_{{\rm BMO}^{q_{2}}_{X}(\omega)}|Q|\lambda^{-q_{2}}d\lambda\\
&=&\omega(Q)N^{q_{1}}+\frac{q_{1}}{q_{2}-q_{1}}\|f\|^{q_{2}}_{{\rm BMO}_{X}(\omega)}\omega(Q)N^{q_{1}-q_{2}},
\end{eqnarray*}
which gives
$$\bigg(\frac{1}{\omega(Q)}\int_{Q}|f(y)-f_{Q}|^{q_{1}}\omega(x)^{1-q_{1}}dy\bigg)^{1/q_{1}}\leq 2\Big(\frac{q_{1}}{q_{2}-q_{1}}\Big)^{1/q_{2}}\|f\|_{{\rm BMO}_{X}(\omega)}.$$
Then
$$\|f\|_{{\rm BMO}^{q_{1}}(\omega)}\leq 2\Big(\frac{q_{1}}{q_{2}-q_{1}}\Big)^{1/q_{2}}\|f\|_{{\rm BMO}_{X}(\omega)}$$
and the lemma follows.
\end{proof}

Let $\omega\in A_{1}$ and $d\mu(x)=\omega(x)dx$. For $0<r<\infty$, we set
$$\|f\|_{{\rm BMO}_{r}(\omega)}=\sup_{Q}\inf_{c}\Big\{\frac{1}{\mu(Q)}\int_{Q}\Big(\frac{|f(x)-c|}{\omega(x)}\Big)^{r}d\mu(x)\Big\}^{1/r},$$
and ${\rm BMO}_{r}(\omega)=\{f\in L_{loc}:\|f\|_{{\rm BMO}_{r}(\omega)}<\infty\}$.
\begin{lemma}\label{lem2}
Let $0<r<1$, $\omega\in A_{1}$ and $d\mu(x)=\omega(x)dx$. Suppose $\|f\|_{{\rm BMO}_{r}(\omega)}=1$ and for each cube $Q$ let $c_{Q}$ be the value which minimizes $\displaystyle{\int_{Q}\Big(\frac{|f(x)-c|}{\omega(x)}\Big)^{r}d\mu(x)}$. Then
$$\mu\Big(\big\{x\in Q: \frac{|f(x)-c_{Q}|}{\omega(x)}>t\big\}\Big)\leq c_{1}e^{-c_{2}t}\mu(Q),$$
where $c_{1}$ and $c_{2}$ are positive constants.
\end{lemma}
\begin{proof}
Take any cube $Q$, write $E_{Q}=\{x\in Q: \frac{|f(x)-c_{Q}|}{\omega(x)}>t\}$. Then
\begin{eqnarray*}
\mu(E_{Q})&\leq& \int_{E_{Q}}\frac{|f(x)-c_{Q}|^{r}}{t^{r}\omega(x)^{r}}d\mu(x)\\
&\leq&\frac{1}{t^{r}}\frac{\mu(Q)}{\mu(Q)}\int_{Q}\frac{|f(x)-c_{Q}|^{r}}{\omega(x)^{r}}d\mu(x)\\
&\leq&\frac{1}{t^{r}}\mu(Q).
\end{eqnarray*}
Write $F_{1}(t)=\frac{1}{t^{r}}$, then
$$\mu(E_{Q})\leq F_{1}(t)\mu(Q).$$

Let $s>1$ and $t\in (0,\infty)$ such that $2^{\frac{n+1}{r}}s[\omega]_{A_{1}}^{1/r}\leq t$. Fix a cube $Q_{0}$, there is a Calderon-Zygmund decomposition of disjoint cubes $\{Q_{j}\}$ such that $Q_{j}\subset Q_{0}$ and\\

~~$(i)~~ \displaystyle{s^{r}<\frac{1}{\mu(Q_{j})}\int_{Q_{j}}\Big(\frac{|f(x)-c_{Q_{0}}|}{\omega(x)}\Big)^{r}d\mu(x)\leq 2^{n}s^{r}},$\\

~~$(ii)~~\displaystyle{\frac{|f(x)-c_{Q_{0}}|}{\omega(x)}\leq s}$ for $x\in \big(\bigcup_{j}Q_{j}\big)^{c}$.

Since $\omega\in A_{1}$ and $x\in Q$, then $\frac{1}{\omega(x)}\leq\frac{[\omega]_{A_{1}}|Q|}{\mu(Q)}$. By $(i)$ and $0<r<1$,
\begin{eqnarray*}
\int_{Q_{j}}\Big(\frac{|f(y)-c_{Q_{0}}|}{\omega(y)}\Big)^{r}\omega(y)^{r}dy
&=&\int_{Q_{j}}\Big(\frac{|f(y)-c_{Q_{0}}|}{\omega(y)}\Big)^{r}\omega(y)^{r-1}d\mu(y)\\
&\leq&\Big(\frac{[\omega]_{A_{1}}|Q_{j}|}{\mu(Q_{j})}\Big)^{1-r}\int_{Q_{j}}\Big(\frac{|f(y)-c_{Q_{0}}|}{\omega(y)}\Big)^{r}d\mu(y)\\
&\leq&2^{n}s^{r}[\omega]^{1-r}_{A_{1}}|Q_{j}|^{1-r}\mu(Q_{j})^{r}.
\end{eqnarray*}
Notice that
$$\int_{Q_{j}}\Big(\frac{|f(x)-c_{Q_{j}}|}{\omega(x)}\Big)^{r}d\mu(x)\leq \int_{Q_{j}}\Big(\frac{|f(x)-c_{Q_{0}}|}{\omega(x)}\Big)^{r}d\mu(x),$$
which implies that
$$\int_{Q_{j}}\Big(\frac{|f(y)-c_{Q_{j}}|}{\omega(y)}\Big)^{r}\omega(y)^{r}dy\leq 2^{n}s^{r}[\omega]^{1-r}_{A_{1}}|Q_{j}|^{1-r}\mu(Q_{j})^{r}.$$
Therefore,\begin{eqnarray*}
\Big(\frac{|c_{Q_{j}}-c_{Q_{0}}|}{\omega(x)}\Big)^{r}
&=&\frac{\omega(x)^{-r}}{|Q_{j}|}\int_{Q_{j}}|c_{Q_{j}}-c_{Q_{0}}|^{r}dy\\
&\leq&\frac{1}{|Q_{j}|\omega(x)^{r}}\int_{Q_{j}}\Big(\frac{|f(y)-c_{Q_{j}}|}{\omega(y)}\Big)^{r}\omega(y)^{r}dy\\
&&+\frac{1}{|Q_{j}|\omega(x)^{r}}\int_{Q_{j}}\Big(\frac{|f(y)-c_{Q_{0}}|}{\omega(y)}\Big)^{r}\omega(y)^{r}dy\\
&\leq&2^{n+1}s^{r}[\omega]^{1-r}_{A_{1}}\Big(\frac{\mu(Q_{j})}{|Q_{j}|\omega(x)}\Big)^{r}\\
&\leq&2^{n+1}s^{r}[\omega]_{A_{1}}.
\end{eqnarray*}

From the fact that $[\omega]_{A_{1}}\geq 1$, we have $t>s$, by $(i)$ and $(ii)$, we have
\begin{eqnarray*}
\mu(E_{Q_{0}})&=&\sum_{j}\mu\big(\{x\in Q_{j}:\frac{|f(x)-c_{Q_{0}}|}{\omega(x)}>t\}\big)\\
&\leq&\sum_{j}\mu\big(\{x\in Q_{j}:\frac{|f(x)-c_{Q_{j}}|}{\omega(x)}+\frac{|c_{Q_{j}}-c_{Q_{0}}|}{\omega(x)}>t\}\big)\\
&\leq&\sum_{j}\mu\big(\{x\in Q_{j}:\frac{|f(x)-c_{Q_{j}}|}{\omega(x)}>t-2^{\frac{n+1}{r}}s[\omega]^{1/r}_{A_{1}}\}\big)\\
&\leq&\sum_{j}F_{1}(t-2^{\frac{n+1}{r}}s[\omega]^{1/r}_{A_{1}})\cdot \mu(Q_{j})\\
&\leq&F_{1}(t-2^{\frac{n+1}{r}}s[\omega]^{1/r}_{A_{1}})\sum_{j}\frac{1}{s^{r}}\int_{Q_{j}}\Big(\frac{|f(x)-c_{Q_{0}}|}{\omega(x)}\Big)^{r}d\mu(x)\\
&\leq&\frac{F_{1}(t-2^{\frac{n+1}{r}}s[\omega]^{1/r}_{A_{1}})}{s^{r}}\int_{Q_{0}}\Big(\frac{|f(x)-c_{Q_{0}}|}{\omega(x)}\Big)^{r}d\mu(x)\\
&\leq&\frac{F_{1}(t-2^{\frac{n+1}{r}}s[\omega]^{1/r}_{A_{1}})}{s^{r}}\mu(Q_{0}).
\end{eqnarray*}
Let
$$F_{2}(t)= \frac{F_{1}(t-2^{\frac{n+1}{r}}s[\omega]^{1/r}_{A_{1}})}{s^{r}}.$$
Continue this process indefinitely, we obtain for any $k\geq 2$,
$$F_{k}(t)=\frac{F_{k-1}(t-2^{\frac{n+1}{r}}s[\omega]^{1/r}_{A_{1}})}{s^{r}},$$
and
$$\mu(E_{Q_{0}})\leq F_{k}(t)\mu(Q_{0}).$$

We fix a constant $t>0$. If
$$k\cdot2^{\frac{n+1}{r}}[\omega]^{1/r}_{A_{1}}s<t\leq (k+1)\cdot2^{\frac{n+1}{r}}[\omega]^{1/r}_{A_{1}}s.$$
for some $k\geq 1$, thus
\begin{eqnarray*}
\mu(E_{Q_{0}})&\leq& \mu\big(\{x\in Q_{0}:\frac{|f(x)-c_{Q_{0}}|}{\omega(x)}>t\}\big)\\
&\leq&\mu\big(\{x\in Q_{0}:\frac{|f(x)-c_{Q_{0}}|}{\omega(x)}>k\cdot2^{\frac{n+1}{r}}[\omega]^{1/r}_{A_{1}}s\}\big)\\
&\leq&F_{k}(k\cdot2^{\frac{n+1}{r}}[\omega]^{1/r}_{A_{1}}s)\mu(Q_{0})\\
&=&\frac{F_{1}(2^{\frac{n+1}{r}}[\omega]^{1/r}_{A_{1}}s)}{s^{(k-1)r}}\mu(Q_{0})\\
&=&\frac{1}{2^{n+1}[\omega]_{A_{1}}s^{kr}}\mu(Q_{0})\\
&\leq&\frac{e^{-kr\log s}}{2^{n+1}[\omega]_{A_{1}}}\mu(Q_{0})\\
&\leq&\frac{e^{r\log s}}{2^{n+1}[\omega]_{A_{1}}}\exp\Big({-\frac{tr\log s}{2^{\frac{n+1}{r}}[\omega]^{1/r}_{A_{1}}s}}\Big)\mu(Q_{0})\\
\end{eqnarray*}
Since $-k\leq 1-\frac{t}{2^{\frac{n+1}{r}}[\omega]^{1/r}_{A_{1}}s}$. If $t\leq 2^{\frac{n+1}{r}}[\omega]^{1/r}_{A_{1}}s$, then use the trivial estimate
$$\mu(E_{Q_{0}})\leq \mu(Q_{0})\leq e^{-t}e^{2^{\frac{n+1}{r}}[\omega]^{1/r}_{A_{1}}s}\mu(Q_{0}).$$
Recall that $s$ is any real number larger than 1. Choosing $s=e$, this yields
$$\mu\Big(\big\{x\in Q: \frac{|f(x)-c_{Q}|}{\omega(x)}>t\big\}\Big)\leq c_{1}e^{-c_{2}t}\mu(Q),$$
for some positive constants $c_{1}$ and $c_{2}$, which proves the inequality of the lemma \ref{lem2}.
\end{proof}

\begin{lemma}\label{lem2.1}
Let $\omega\in A_{1}$ and $0<r<1$. Then
$${\rm BMO}^{r}(\omega)={\rm BMO}_{r}(\omega).$$
The norms are mutually equivalent.
\end{lemma}
\begin{proof}
By Lemma \ref{lem2} and the homogeneity of $\|\cdot\|_{{\rm BMO}_{r}(\omega)}$, we obtain that for any $f\in {\rm BMO}_{r}(\omega)$,
$$\omega\big(\{x\in Q: \frac{|f(x)-c_{Q}|}{\omega(x)}>t\}\big)\leq c_{1}\exp{(-c_{2}t/\|f\|_{{\rm BMO}_{r}(\omega)})}\omega(Q).$$
This gives us
\begin{eqnarray*}
\frac{1}{\omega(Q)}\int_{Q}|f(x)-c_{Q}|dx
&=&\frac{1}{\omega(Q)}\int_{0}^{\infty}\omega\big(\{x\in Q: \frac{|f(x)-c_{Q}|}{\omega(x)}>t\}\big)dt\\
&\leq&\frac{1}{\omega(Q)}\int_{0}^{\infty}c_{1}\exp{(-c_{2}t/\|f\|_{{\rm BMO}_{r}(\omega)})}\omega(Q)dt\\
&\leq&C\|f\|_{{\rm BMO}_{r}(\omega)}.
\end{eqnarray*}
Therefore,
\begin{eqnarray*}
\frac{1}{\omega(Q)}\int_{Q}\Big(\frac{|f(x)-f_{Q}|}{\omega(x)}\Big)^{r}\omega(x)dx
&\leq&\frac{1}{\omega(Q)}\int_{Q}\Big(\frac{|f(x)-c_{Q}|}{\omega(x)}\Big)^{r}\omega(x)dx\\
&&+\frac{1}{\omega(Q)}\int_{Q}\Big(\frac{|c_{Q}-f_{Q}|}{\omega(x)}\Big)^{r}\omega(x)dx\\
&\leq&\|f\|^{r}_{{\rm BMO}_{r}(\omega)}+\bigg(\frac{1}{\omega(Q)}\int_{Q}|f(x)-c_{Q}|dx\bigg)^{r}\\
&\leq&C\|f\|^{r}_{{\rm BMO}_{r}(\omega)}.
\end{eqnarray*}
Conversely, $\|\cdot\|_{{\rm BMO}_{r}(\omega)}\leq \|\cdot\|_{{\rm BMO}^{r}(\omega)}$ is obvious. Thus, the equivalence of $\|\cdot\|_{{\rm BMO}_{r}(\omega)}$ and $\|\cdot\|_{{\rm BMO}^{r}(\omega)}$ is shown.
\end{proof}
\vspace{0.5cm}

Standard real analysis tools as the maximal function $M(f)$, the weighted maximal function $M_{\omega}(f)$ and the sharp maximal function $M^{\sharp}(f)$ carries over to this context, namely,
$$M(f)(x)=\sup_{Q\ni x}\frac{1}{|Q|}\int_{Q}|f(y)|dy;$$
$$M_{\omega}(f)(x)=\sup_{Q\ni x}\frac{1}{\omega(Q)}\int_{Q}|f(y)|\omega(y)dy;$$
$$M^{\sharp}(f)(x)=\sup_{Q\ni x}\inf_{c}\frac{1}{|Q|}\int_{Q}|f(y)-c|dy\approx \sup_{Q\ni x}\frac{1}{|Q|}\int_{Q}|f(y)-f_{Q}|dy.$$
A variant of weighted maximal function and sharp maximal operator $M_{\omega,s}(f)(x)=\big(M_{\omega}(f^{s})\big)^{1/s}$ and $M_{\delta}^{\sharp}(f)(x)=\big(M^{\sharp}(f^{\delta})(x)\big)^{1/\delta}$, which will become the main tool in our scheme.

The following relationships between $M_{\delta}$ and $M_{\delta}^{\sharp}$ to be used is a version of the classical ones
due to Fefferman and Stein \cite{FS}.
\begin{lemma}\label{lem3}
Let $0<p,\delta<\infty$ and $\omega\in A_{\infty}$. There exists a positive $C$ such that
$$\int_{\mathbb{R}^n}(M_{\delta}f(x))^{p}\omega(x)dx\leq C\int_{\mathbb{R}^n}(M^{\sharp}_{\delta}f(x))^{p}\omega(x)dx,$$
for any smooth function $f$ for which the left-hand side is finite.
\end{lemma}

\begin{lemma}\label{lem4}
Let $\omega\in A_{1}$ and $b\in {\rm BMO}(\omega)$. Then, there exists a constant $C$ such that
\begin{eqnarray}\label{eq1}
\begin{split}
M^{\sharp}_{\frac{1}{2}}\big([b,T]_{1}(f_{1},f_{2})\big)(x)
&\leq C\|b\|_{{\rm BMO}(\omega)}\omega(x)M(T(f_{1},f_{2})(x))\\
&\quad+C\|b\|_{{\rm BMO}(\omega)}\omega(x)M_{\omega,s}(f_{1})(x)M(f_{2})(x),
\end{split}
\end{eqnarray}

\begin{eqnarray}\label{eq2}
\begin{split}
M^{\sharp}_{\frac{1}{2}}\big([b,T]_{2}(f_{1},f_{2})\big)(x)
&\leq C\|b\|_{{\rm BMO}(\omega)}\omega(x)M(T(f_{1},f_{2})(x))\\
&\quad+C\|b\|_{{\rm BMO}(\omega)}\omega(x)M(f_{1})(x)M_{\omega,s}(f_{2})(x),
\end{split}
\end{eqnarray}
and
\begin{eqnarray}\label{eq3}
\begin{split}
M^{\sharp}_{1/3}\big([\Pi\vec{b},T](f_{1},f_{2})\big)(x)&\leq C \omega(x)^{2}\|b\|^{2}_{{\rm BMO}(\omega)}M\big(T(f_{1},f_{2})\big)(x)\\
&\quad+C\omega(x)\|b\|_{{\rm BMO}(\omega)}M_{1/2}\big([b,T]_{1}(f_{1},f_{2})\big)(x)\\
&\quad+C\omega(x)\|b\|_{{\rm BMO}(\omega)}M_{1/2}\big([b,T]_{2}(f_{1},f_{2})\big)(x)\\
&\quad+C\omega(x)^{2}\|b\|^{2}_{{\rm BMO}(\omega)}M_{\omega,s}(f_{1})(x)M_{\omega,s}(f_{2})(x),
\end{split}
\end{eqnarray}
for any $1<s<\infty$ and bounded compact supported functions $f_{1},f_{2}$.
\end{lemma}

\begin{proof}
we only prove (\ref{eq1}) and the proof of (\ref{eq2}) and (\ref{eq3}) are very similar to that of (\ref{eq1}). Let $Q:=Q(x_{0},r)$ be a cube and $x\in Q$. Then,
\begin{eqnarray*}
&&\bigg(\frac{1}{|Q|}\int_{Q}\Big|\big|[b,T]_{1}(f_{1},f_{2})(z)\big|^{1/2}-|c|^{1/2}\Big|dz\bigg)^{2}\\
&&\leq C\bigg(\frac{1}{|Q|}\int_{Q}\big|[b,T]_{1}(f_{1},f_{2})(z)-c\big|^{1/2}dz\bigg)^{2}\\
&&\leq C\bigg(\frac{1}{|Q|}\int_{Q}\big|(b(z)-\lambda)T(f_{1},f_{2})(z)\big|^{1/2}dz\bigg)^{2}\\
&&\quad+\bigg(\frac{1}{|Q|}\int_{Q}\big|T((b-\lambda)f_{1},f_{2})(z)-c\big|^{1/2}dz\bigg)^{2}\\
&&=:A_{1}+A_{2},
\end{eqnarray*}
where $\lambda=b_{2Q}$.

We first consider the term $A_{1}$. By H\"{o}lder inequality, we obtain that
\begin{eqnarray*}
A_{1}&=&\bigg(\frac{1}{|Q|}\int_{Q}\Big|(b(z)-\lambda)T(f_{1},f_{2})(z)\Big|^{1/2}dz\bigg)^{2}\\
&\leq& C\|b\|_{{\rm BMO}(\omega)}\frac{\omega(Q)}{|Q|}
\cdot\frac{1}{|Q|}\int_{Q}\big|T(f_{1},f_{2})(z)\big|dz\\
&\leq& C\omega(x)\|b\|_{{\rm BMO}(\omega)}M\big(T(f_{1},f_{2})\big)(x).
\end{eqnarray*}

Let us consider next the term $A_{2}$. Let
$$\Omega_{0}=\{(y_{1},y_{2})\in\mathbb{R}^n\times\mathbb{R}^n: |x_{0}-y_{1}|+|x_{0}-y_{2}|\leq 2\sqrt{n}r\}$$
and for $k\geq 1$,
$$\Omega_{k}=\{(y_{1},y_{2})\in\mathbb{R}^n\times\mathbb{R}^n: 2^{k+1}\sqrt{n}\geq |x_{0}-y_{1}|+|x_{0}-y_{2}|> 2^{k}\sqrt{n}r\}.$$
We write
\begin{eqnarray*}
A_{2}&\leq &\bigg(\frac{1}{|Q|}\int_{Q}
\bigg|\iint_{\Omega_{0}}(b(y_{1})-\lambda)K(z-y_{1},z-y_{2})f_{1}(y_{1})f_{2}(y_{2})dy_{1}dy_{2}\bigg|^{1/2}dz\bigg)^{2}\\
&&+\bigg(\frac{1}{|Q|}\int_{Q}\bigg|\iint_{\mathbb{R}^n\times\mathbb{R}^n\backslash\Omega_{0}}
(b(y_{1})-\lambda)K(z-y_{1},z-y_{2})f_{1}(y_{1})f_{2}(y_{2})dy_{1}dy_{2}-c\bigg|^{1/2}dz\bigg)^{2}\\
&=:&A_{21}+A_{22}.
\end{eqnarray*}

It is obvious that $\Omega_{0}\subset 4\sqrt{n}Q\times 4\sqrt{n}Q$, we write $f^{0}_{i}=f_{i}\chi_{4\sqrt{n}Q}$. By Kolmogorov inequality and the fact that $T$ is bounded from $L^{1}\times L^{1}$ to $L^{1/2,\infty}$, we get
\begin{eqnarray*}
A_{21}&\leq &\frac{1}{|Q|^{2}}\|T((b-b_{Q})f^{0}_{1},f^{0}_{2})\|_{L^{1/2,\infty}}\\
&\leq &\frac{C}{|Q|^{2}}\int_{4\sqrt{n}Q}|b(y_{1})-b_{Q}||f_{1}(y_{1})|dy_{1}\int_{4\sqrt{n}Q}|f_{2}(y_{2})|dy_{2}\\
&\leq &C\omega(x)\|b\|_{{\rm BMO}^{s'}(\omega)}M_{\omega,s}(f_{1})(x)M(f_{2})(x).
\end{eqnarray*}

Let
$$c=\frac{1}{|Q|}\int_{Q}\iint_{\mathbb{R}^n\times\mathbb{R}^n\backslash \Omega_{0}}(b(y_{1})-\lambda)K(z'-y_{1},z'-y_{2})f_{1}(y_{1})f_{2}(y_{2})dz'$$
For any $z,z'\in Q$ and $y_{1},y_{2}$ such that $|x_{0}-y_{1}|+|x_{0}-y_{2}|> 2\sqrt{n}r$, then
$$|z-z'|\leq \frac{1}{2}\sqrt{n}r\leq \frac{1}{2}\max\{|x_{0}-y_{1}|,|x_{0}-y_{2}|\},$$
which gives us that
$$\Big|K(z-y_{1},z-y_{2})-K(z'-y_{1},z'-y_{2})\Big|\leq \frac{C|z-z'|^{\gamma}}{\big(|z-y_{1}|+|z-y_{2}|\big)^{2n+\gamma}}.$$
Therefore,
\begin{eqnarray*}
A_{22}&\leq&C \sum_{k=1}^{\infty}\frac{r^{\gamma}}{|Q|}\int_{Q}\iint_{\Omega_{k}}
\frac{|b(y_{1})-\lambda||f_{1}(y_{1})||f_{2}(y_{2})|}{\big(|z-y_{1}|+|z-y_{2}|\big)^{2n+\gamma}}dy_{1}dy_{2}dz\\
&\leq& C\sum_{k=1}^{\infty}\Big(\frac{1}{2^{kn}}\Big)^{\gamma}
\cdot\frac{1}{|2^{k}Q|^{2}}\int_{2^{k+1}\sqrt{n}Q}\int_{2^{k+1}\sqrt{n}Q}|b(y_{1})-\lambda||f_{1}(y_{1})||f_{2}(y_{2})|dy_{1}dy_{2}\\
&\leq& C\omega(x)\|b\|_{{\rm BMO}^{s'}(\omega)}M_{\omega,s}(f_{1})(x)M(f_{2})(x).
\end{eqnarray*}

Collecting our estimates, we have shown that
\begin{eqnarray*}
M^{\sharp}_{\frac{1}{2}}\big([b,T]_{1}(f_{1},f_{2})\big)(x)&\leq & C\|b\|_{{\rm BMO}(\omega)}\omega(x)M(T(f_{1},f_{2}))(x)\\
&&+C\|b\|_{{\rm BMO}(\omega)}\omega(x)M_{\omega,s}(f_{1})(x)M(f_{2})(x)
\end{eqnarray*}
for any $1<s<\infty$ and bounded compact supported functions $f_{1},f_{2}$.
\end{proof}

\section{Proof of Theorem \ref{main1} $\sim$ Theorem \ref{t1}}

{\it Proof of Theorem \ref{main1}.}
Let $1<p<\infty$, $\omega\in A_{1}$ and $X=L^{p,\infty}(\omega)$. By Lemma \ref{lem1}, we have
$$\|\cdot\|_{{\rm BMO}(\omega)}\leq C\|\cdot\|_{{\rm BMO}_{X}(\omega)}.$$
From the fact that ${\rm BMO}(\omega)={\rm BMO}^{p}(\omega)$ and $\|\cdot\|_{{\rm BMO}_{X}(\omega)}\leq \|\cdot\|_{{\rm BMO}^{p}(\omega)}$, it follows that
$$\|\cdot\|_{{\rm BMO}_{X}(\omega)}\leq C\|\cdot\|_{{\rm BMO}(\omega)}.$$
Thus we complete the proof of Theorem \ref{main1}. \qed
\vspace{0.5cm}

{\it Proof of Theorem \ref{main2}.}
Let $f\in {\rm BMO}^{r}(\omega)$. In the proof of lemma \ref{lem2.1}, we have shown that
\begin{eqnarray*}
\frac{1}{\omega(Q)}\int_{Q}|f(x)-c_{Q}|dx
\leq C\|f\|_{{\rm BMO}_{r}(\omega)}.
\end{eqnarray*}
Therefore,
\begin{eqnarray*}
\frac{1}{\omega(Q)}\int_{Q}|f(x)-f_{Q}|dx&\leq&
\frac{2}{\omega(Q)}\int_{Q}|f(x)-c_{Q}|dx\\
&\leq& C\|f\|_{{\rm BMO}_{r}(\omega)}\leq C\|f\|_{{\rm BMO}^{r}(\omega)}.
\end{eqnarray*}
As a result, $\|f\|_{{\rm BMO}(\omega)}\leq C\|f\|_{{\rm BMO}^{r}(\omega)}$. The opposite inequality is a consequence of H\"{o}lder inequality, then the equivalence of $\|f\|_{{\rm BMO}(\omega)}$ and $\|f\|_{{\rm BMO}^{r}(\omega)}$ is shown. \qed

\vspace{0.5cm}
{\it Proof of Theorem \ref{t1}.}
$(a1)\Rightarrow (a2)$: Since $\omega\in A_{1}$, then $\omega^{1-p}\in A_{\infty}$. By Lemma \ref{lem3} and Lemma \ref{lem4} with $1<s<\min\{p_{1},p_{2}\}$, from a standard argument that we can obtain
\begin{eqnarray*}
\|[\Sigma \vec{b},T](f_{1},f_{2})\omega^{-1}\|_{L^{p}(\omega)}
&=&\|[\Sigma \vec{b},T](f_{1},f_{2})\|_{L^{p}(\omega^{1-p})}\\
&\leq&\|M_{\frac{1}{2}}\big([\Sigma \vec{b},T](f_{1},f_{2})\big)\|_{L^{p}(\omega^{1-p})}\\
&\leq&C\|M^{\sharp}_{\frac{1}{2}}\big([\Sigma \vec{b},T](f_{1},f_{2})\big)\|_{L^{p}(\omega^{1-p})}\\
&\leq&C\|b\|_{{\rm BMO}(\omega)}\big\|M\big(T(f_{1},f_{2})\big)\big\|_{L^{p}(\omega)}\\
&&+C\|b\|_{{\rm BMO}(\omega)}\|M(f_{1})(x)M_{\omega,s}(f_{2})\|_{L^{p}(\omega)}\\
&\leq&C\|b\|_{{\rm BMO}(\omega)}\prod_{i=1}^{2}\|f_{i}\|_{L^{p_{i}}(\omega)}.
\end{eqnarray*}

We observe that to use the Fefferman-Stein inequality, one needs to verify that certain terms in the left-hand side of the inequalities are finite. We can assume that $f_{1},f_{2}$ are bounded functions with compact support, applying a similar argument as in \cite[pp.32-33]{LOPTT} and Fatou's lemma, one gets the desired result.

$(a2)\Rightarrow (a3)$ is obvious.

$(a3)\Rightarrow (a1)$: Let $z_{0}\in \mathbb{R}^n$ such that $|(z_{0},z_{0})|>2\sqrt{n}$ and let $\delta\in (0,1)$ small enough. Take $B=B\big((z_{0},z_{0}),\delta\sqrt{2n}\big)\subset \mathbb{R}^{2n}$ be the ball for which we can express $\frac{1}{K}$ as an absolutely convergent Fourier series of the form
$$\frac{1}{K(y_{1},y_{2})}=\sum_{j}a_{j}e^{iv_{j}\cdot(y_{1},y_{2})}, \quad (y_{1},y_{2})\in B,$$
with $\sum_{j}|a_{j}|<\infty$ and we do not care about the vectors $v_{j}\in \mathbb{R}^{2n},$ but we will at times express them as $v_{j}=(v_{j}^{1},v_{j}^{2})\in \mathbb{R}^{n}\times \mathbb{R}^n.$

Set $z_{1}=\delta^{-1}z_{0}$ and note that
$$\big(|y_{1}-z_{1}|^{2}+|y_{2}-z_{1}|^{2}\big)^{1/2}<\sqrt{2n}\Rightarrow \big(|\delta y_{1}-z_{0}|^{2}+|\delta y_{2}-z_{0}|^{2}\big)^{1/2}<\delta \sqrt{2n}.$$
Then for any $(y_{1},y_{2})$ satisfying the inequality on the left, we have
$$\frac{1}{K(y_{1},y_{2})}=\frac{\delta^{-2n}}{K(\delta y_{1},\delta y_{2})}=\delta^{-2n}\sum_{j}a_{j}e^{i\delta v_{j}\cdot(y_{1},y_{2})}.$$

Let $Q=Q(x_{0},r)$ be any arbitrary cube in $\mathbb{R}^n$. Set $\tilde{z}=x_{0}+rz_{1}$ and take $Q'=Q(\tilde{z},r)\subset \mathbb{R}^n$. So for any $x\in Q$ and $y_{1},y_{2}\in Q'$, we have
$$\Big|\frac{x-y_{1}}{r}-z_{1}\Big|\leq \Big|\frac{x-x_{0}}{r}\Big|+\Big|\frac{y_{1}-\tilde{z}}{r}\Big|\leq \sqrt{n},
\quad \Big|\frac{x-y_{2}}{r}-z_{1}\Big|\leq \Big|\frac{x-x_{0}}{r}\Big|+\Big|\frac{y_{2}-\tilde{z}}{r}\Big|\leq \sqrt{n},$$
which implies that
$$\bigg(\Big|\frac{x-y_{1}}{r}-z_{1}\Big|^{2}+\Big|\frac{x-y_{2}}{r}-z_{1}\Big|^{2}\bigg)^{1/2}\leq \sqrt{2n}.$$
\vspace{0.3cm}

Let $s(x)=\overline{\mathrm{sgn}(\int_{Q'}(b(x)-b(y))dy)}$. Then
\begin{eqnarray}\label{eq4}
\begin{split}
|b(x)-b_{Q'}|&=s(x)\big(b(x)-b_{Q'}\big)\\
&=\frac{s(x)}{|Q'|^{2}}\int_{Q'}\int_{Q'}\big(b(x)-b(y_{1})\big)\big)dy_{1}dy_{2}\\
\end{split}
\end{eqnarray}
Setting
$$g_{j}(y_{1})=e^{-i\frac{\delta}{r}v^{1}_{j}\cdot y_{1}}\chi_{Q'}(y_{1}),$$
$$h_{j}(y_{2})=e^{-i\frac{\delta}{r}v^{2}_{j}\cdot y_{2}}\chi_{Q'}(y_{2}),$$
$$m_{j}(x)=e^{i\frac{\delta}{r}v_{j}\cdot (x,x)}\chi_{Q}(x)s(x),$$
which shows that
\begin{eqnarray*}
|b(x)-b_{Q'}|&=& s(x)\frac{r^{2n}\delta^{-2n}}{|Q'|^{2}}
\int_{Q'}\int_{Q'}\frac{b(x)-b(y_{1})}
{\big(|x-y_{1}|^{2}+|x-y_{2}|^{2}\big)^{n-\alpha/2}}\\
&&\qquad\times\sum_{j}a_{j}e^{i\frac{\delta}{r}v_{j}\cdot(x-y_{1},x-y_{2})}dy_{1}dy_{2}\\
&&=\sum_{j}a_{j}[b,T]_{1}(g_{j},h_{j})(x)m_{j}(x).
\end{eqnarray*}

If $p>1$, we have the following estimate
\begin{eqnarray*}
&&\frac{\lambda}{\omega(Q)^{1/p}}\omega\big(x\in Q:\frac{|b(x)-b_{Q'}|}{\omega(x)}>\lambda\big)^{1/p}\\
&&=\frac{\lambda}{\omega(Q)^{1/p}}\omega\big(x\in Q:\frac{|b(x)-b_{Q'}|}{\omega(x)}>\lambda\big)^{1/p}\\
&&\leq \frac{\lambda}{\omega(Q)^{1/p}}\omega\big(x\in Q:\frac{\sum_{j}|a_{j}|\big|[b,T]_{1}(g_{j},h_{j})(x)\big|}{\omega(x)}>\lambda\big)^{1/p}\\
&&\leq \frac{C}{\omega(Q)^{1/p}}\sum_{j}|a_{j}|\|[b,T]_{1}(g_{j},h_{j})\|_{L^{p,\infty}(\omega)}\\
&&\leq C\sum_{j}|a_{j}|.
\end{eqnarray*}
We write
$$\|b\|_{{\rm BMO}_{*}(\omega)}:=\sup_{Q}\sup_{\lambda>0}\frac{\lambda}{\omega(Q)^{1/p}}\omega\big(x\in Q:\frac{|b(x)-b_{Q'}|}{\omega(x)}>\lambda\big)^{1/p},$$
then $\|b\|_{{\rm BMO}_{*}(\omega)}\leq C\sum_{j}|a_{j}|$. The same estimate as lemma \ref{lem1}, we conclude that
\begin{eqnarray*}
|b_{Q}-b_{Q'}|&\leq& \frac{1}{|Q|}\int_{Q}|b(x)-b_{Q'}|dx\\
&\leq& \frac{\omega(Q)}{|Q|}\|b\|_{{\rm BMO}_{*}(\omega)}\\
&\leq& C\frac{\omega(Q)}{|Q|}\sum_{j}|a_{j}|.
\end{eqnarray*}
By the definition of $A_{1}$ weights, we concluded that $\omega(Q)\leq |Q|\omega(x)$, which implies that for any cube $Q$ and $\lambda>0$,
\begin{eqnarray*}
&&\frac{\lambda}{\omega(Q)^{1/p}}\omega\big(x\in Q:\frac{|b(x)-b_{Q}|}{\omega(x)}>\lambda\big)^{1/p}\\
&&\leq\frac{\lambda}{\omega(Q)^{1/p}}\omega\big(x\in Q:\frac{|b(x)-b_{Q'}|}{\omega(x)}>\frac{\lambda}{2}\big)^{1/p}\\
&&\quad+\frac{\lambda}{\omega(Q)^{1/p}}\omega\big(x\in Q:\frac{|b_{Q'}-b_{Q}|}{\omega(x)}>\frac{\lambda}{2}\big)^{1/p}\\
&&\leq C\sum_{j}|a_{j}|+\frac{\lambda}{\omega(Q)^{1/p}}\omega\big(x\in Q:\frac{C\sum_{j}|a_{j}|\omega(Q)}{|Q|\omega(x)}>\frac{\lambda}{2}\big)^{1/p}\\
&&\leq C\sum_{j}|a_{j}|.
\end{eqnarray*}
This shows that $b\in {\rm BMO}_{X}(\omega)$ with $X=L^{p,\infty}$; that is, the symbol $b$ belongs to ${\rm BMO}(\omega)$.

If $p\leq 1$, choose $q\in (0,p)$. By the fact that $L^{p,\infty}(\omega)\subset M^{p}_{q}(\omega)$ in \cite[Corollary 2.3]{WZC} (see also \cite[Lemma 1.7]{KY} for the unweighted case), $M^{p}_{q}(\omega)$ stands for the weighted Morrey spaces; that is, for $0<q<p<\infty$,
$$M^{p}_{q}(\omega)=\bigg\{f\in L^{q_{1}}_{\it loc}:\|f\|_{M^{p}_{q}}=\sup_{Q}\frac{1}{\omega(Q)^{1/q-1/p}}\Big(\int_{Q}|f(y)|^{q}\omega(y) dy\Big)^{1/q}<\infty\bigg\}.$$
Therefore,
\begin{eqnarray*}
&&\inf_{c}\bigg(\frac{1}{\omega(Q)}\int_{Q}\Big(\frac{|b(x)-c|}{\omega(x)}\Big)^{q}\omega(x)dx\bigg)^{1/q}\\
&&\leq\bigg(\frac{1}{\omega(Q)}\int_{Q}\Big(\frac{|b(x)-b_{Q'}|}{\omega(x)}\Big)^{q}\omega(x)dx\bigg)^{1/q}\\
&&\leq \bigg(\frac{C}{\omega(Q)}\int_{Q}\Big|\sum_{j}|a_{j}|[b, T]_{1}(g_{j},h_{j})(x)\omega(x)^{-1}\Big|^{q}\omega(x)dx\bigg)^{1/q}\\
&&\leq C\omega(Q)^{-1/p}\sum_{j}|a_{j}|\|[b, T]_{1}(g_{j},h_{j})\omega^{-1}\|_{M^{p}_{q}(\omega)}\\
&&\leq C\omega(Q)^{-1/p}\sum_{j}|a_{j}|\|[b, T]_{1}(g_{j},h_{j})\omega^{-1}\|_{L^{p,\infty}(\omega)}\\
&&\leq C\sum_{j}|a_{j}|.
\end{eqnarray*}
Thus showing that $b\in {\rm BMO}_{q}(\omega)$. The desired result follows from here.

By the inequality \eqref{eq3} in lemma \ref{lem3} and the same argument as $(a1)\Rightarrow (a2)$, we can obtain that $(a1)\Rightarrow (a4)$. It is easy to see that $(a4)\Rightarrow (a5)$. The proof of $(a5)\Rightarrow (a1)$ follows the method that of $(a3)\Rightarrow (a1)$ except replacing \eqref{eq4} by
\begin{eqnarray*}
|b(x)-b_{Q'}|^{2}&=&s(x)^{2}\big(b(x)-b_{Q'}\big)\big(b(x)-b_{Q'}\big)\\
&=&\frac{s(x)^{2}}{|Q'|^{2}}\int_{Q'}\int_{Q'}\big(b(x)-b(y_{1})\big)\big(b(x)-b(y_{2})\big)dy_{1}dy_{2}.
\end{eqnarray*}

Therefore, we complete the proof of Theorem \ref{t1}. \qed

\color{black}

\end{document}